\documentclass[a4paper]{amsart}
\usepackage[utf8]{inputenc}
\usepackage[english]{babel}
\usepackage{fancyhdr}
\usepackage{amsmath}
\usepackage{amsfonts,amstext}
\usepackage{amssymb,amsthm,amscd,amsxtra,wasysym,graphicx}
\usepackage{mathrsfs,esint,comment}
\usepackage{tikz-cd}
\usepackage{enumitem,mathtools,dsfont}
 \usepackage{palatino,mathpazo}
\usepackage{palatino}
\usepackage{charter}
\usepackage{xcolor}

\def\Bibtex{{\rm B\kern-.05em{\sc i\kern-.025em b}\kern-0.08em T\kern-.1667em\lower.7ex\hbox{E}\kern-.125emX}}
\usepackage[breaklinks=true,bookmarks=false,pagebackref]{hyperref}
\usepackage{hyperref}
\hypersetup{
	unicode=false,          
	pdftoolbar=true,       
	pdfmenubar=true,       
	pdffitwindow=false,     
	pdfstartview={FitH}, 
	pdftitle={Modulus of continuity},   
	pdfauthor={Quang-Tuan Dang},   
	colorlinks=true,  
	linkcolor=purple,         
	citecolor=blue,        
	filecolor=green,      
	urlcolor=blue}

\usepackage{colonequals}

\usepackage{enumitem}

\numberwithin{equation}{section}

\newtheorem{theorem}{Theorem}[section]
\newtheorem{proposition}[theorem]{Proposition}

\newtheorem{lemma}[theorem]{Lemma}
\theoremstyle{definition}
\newtheorem{definition}[theorem]{Definition}
\newtheorem{remark}[theorem]{Remark}

\newtheorem{example}[theorem]{Example}

\newcommand{\Vol}{{\rm Vol}}

\newcommand{\ddc}{dd^c}

\DeclareMathOperator{\ric}{Ric}

\newcommand{\PSH}{{\rm PSH}}

\newcommand{\exph}{{\rm exph}}

\begin{document}
	\title[Uniqueness and Stability]{Uniqueness and Stability of Monge--Amp\`ere potentials in Big Cohomology Classes}
	\author{Quang-Tuan Dang, Lei Zhang and Bin Zhou}
		\address{Yau Mathmatical Sciences Center, Tsinghua University, Beiing 100084}
		\email{dangquangtuan10@gmail.com $\&$ dangqt@mail.tsinghua.edu.cn}
        \address{Yau Mathmatical Sciences Center, Tsinghua University, Beiing 100084}
        \email{leizhang92@mail.tsinghua.edu.cn}
        \address{School of Mathematical Sciences, Peking University, Beijing 100871}
        \email{bzhou@pku.edu.cn}
	
	\date{\today}
	\keywords{Complex Monge-Amp\`ere equations; stability; uniqueness; modulus of continuity}
	\subjclass[2020]{32U15, 32W20, 32Q15}

	\begin{abstract} In this paper, we establish a stability estimate and a uniform estimate for the modulus of continuity of solutions to the degenerate complex Monge--Amp\`ere equation in big cohomology classes. Consequently, we prove the uniqueness of solutions to complex Monge–Ampère mean field equations for a sufficiently small parameter. 
	\end{abstract}

	\maketitle
	
	 \tableofcontents
	\section{Introduction}


Finding canonical metrics in K\"ahler geometry goes back to Yau’s solution of the Calabi conjecture~\cite{yau1978ricci}.  K\"ahler–Einstein (KE) type metrics are of particular interest, and their study has revealed numerous deep connections between differential geometry and algebraic geometry.
In connection with the Minimal Model Program, the search for singular K\"ahler--Einstein metrics naturally leads to the study of degenerate complex Monge--Ampère equations; see~\cite{eyssidieux2009singular,song2012canonical,berman2013variational,berman2014kahler,berman2019kahler} and the references therein.

 Guedj and Zeriahi~\cite{guedj2007weighted} initiated the systematic study of degenerate complex Monge--Ampère equations with very general measures on the right-hand side, developing the finite-energy framework that underlies the theory of non-pluripolar Monge--Ampère measures. This theory was extended to big cohomology classes by Berman, Boucksom, Eyssidieux, Guedj, and Zeriahi~\cite{boucksom2010monge,berman2013variational}.

 \medskip Before stating our results, we fix some notation and terminology.
 Let $X$ be a compact K\"ahler manifold of dimension $n$ equipped with a K\"ahler form $\omega_X$. 
 We let $d$, $d^c$ denote the real differential operators on $X$ defined by $d:=\partial+\Bar{\partial}$, $d^c:=\frac{i}{2}(\Bar{\partial}-\partial)$ so that $\ddc=i\partial\Bar{\partial}$.
	Fix a closed real smooth (1,1)-form $\theta$ representing a big cohomology class. Let $\PSH(X,\theta)$ denote the set of all $\theta$-psh functions. 
	We say that the cohomology class $\{\theta\}$ is {\em big} if $\PSH(X,\theta-\varepsilon\omega_X)\neq \varnothing$ for some $\varepsilon>0$.

    Following \cite{berman2014kahler,Berman-Boucksom-Jonsson}, we say that $\theta_u=\theta+\ddc u$ has a well-defined Ricci curvature if the non-pluripolar self-product $ \theta_u^n $  of $\theta_u$ is well-defined on $X$  (see Section~\ref{sect: nonpluripolar} for this notion) and for every local holomorphic volume form $\Omega$ on $X$, we can write $\theta_u^n= e^{-2f} \omega_X^n$ for some function $f \in L^1_{\rm loc}$. In this case, we define $\ric \theta_u:=\ric \omega_X+ \ddc  f$. Let $\eta$ be a smooth closed (1, 1)-form representing the cohomology class for which the following decomposition holds
    \[ c_1(-K_X)=\{\theta\}+\{\eta\}.\] Let $\psi\in\PSH(X,\eta)$.
	 According to Darvas and Zhang~\cite{Darvas-Zhang24-twisted}, $\theta_u$ is said to be $\eta_\psi$-twisted K\"ahler–Einstein currents/metrics if
\begin{equation}\label{eq: KE}
    \ric \theta_\varphi=\lambda\theta_\varphi+\eta_\psi,\quad\lambda>0.
\end{equation}
We can reduce the twisted K\"ahler–Einstein equation
to the following highly degenerate complex Monge–Amp\`ere equation: 
\begin{equation}\label{eq: tKE}
    (\theta+\ddc \varphi)^n=e^{-\lambda \varphi+\chi-\psi}\omega_X^n,
\end{equation} where $\chi$ is another quasi-psh function on $X$ with analytic singularity type. Note that the left-hand side of~\eqref{eq: tKE} denotes the non-pluripolar Monge--Amp\`ere product; see Section~\ref{sect: nonpluripolar}. We assume that $\chi-\psi$ is klt, i.e., $\int_X e^{\chi-\psi}\omega_X^n<\infty$. For small $\lambda>0$, Guan-Zhou's openness theorem~\cite{guan2015proof} gives $\int_Xe^{-\lambda v+\chi-\psi}\omega_X^n<+\infty$ for all $v\in\mathcal{E}(X,\theta)$, so equation~\eqref{eq: tKE} makes sense. By Guan-Zhou's openness theorem again, there exists $p>1$ such that $e^{\chi-\psi}\in L^p(X,\omega_X^n)$. Darvas and Zhang~\cite[Proposition 5.2.]{Darvas-Zhang24-twisted} showed that there exists a solution $\varphi\in \PSH(X,\theta)$ with minimal singularities to~\eqref{eq: tKE} for some $\lambda>0$ small. Our first main
result establishes the uniqueness of equation~\eqref{eq: tKE}.

\begin{theorem}\label{thm: unique}
     Let $(X,\omega_X)$ be an $n$-dimensional compact K\"ahler manifold and $\theta$ a closed smooth (1,1) form representing a big cohomology class. Then there exists $\lambda_0$, depending
on $X$, $\theta$, $\omega_X$, $n$, $p$, and $\|e^{\chi-\psi}\|_p$, such that for all $\lambda\in(0,\lambda_0)$, the equation \eqref{eq: tKE} has a unique solution $\varphi\in\PSH(X,\theta)$ with minimal singularities.
\end{theorem}
We note that the result in Theorem~\ref{thm: unique} cannot hold for all $\lambda>0$, as shown in~\cite[Example 3.5]{Lu-Phung25-uniqueness}. The proof of Theorem ~\ref{thm: unique} is based on a refinement of the stability estimate of solutions in big cohomology classes obtained in~\cite{kolodziej2003monge,
dinew2010satbility,guedj2018stability,
lu2021stability,Lu-Phung25-uniqueness}. As a consequence, the solution to the twisted K\"ahler-Einstein equation~\eqref{eq: KE} is unique for $\lambda>0$ sufficiently small.

\medskip
Beyond existence and uniqueness, a fundamental problem is to understand the regularity of solutions. When the right-hand side has density in $L^p$ for some $p>1$, the continuity and H\"older regularity of solutions to degenerate complex Monge--Ampère equations have been extensively investigated; see~\cite{kolodziej1998complex,kolodziej2008holder,hiep2010holder,eyssidieux2011viscosity,boucksom2010monge,demailly2014holder,dang2021continuity,DangDoPham26-modulus-big} and the references therein. More generally,  the modulus of continuity of the solution of complex Monge--Ampère equations is closely related to the geometry of the associated K\"ahler metrics. Such estimates, together with suitable control of the Monge--Ampère densities, play an important role in establishing uniform diameter and non-collapsing estimates and in studying Gromov--Hausdorff limits~\cite{Li_yang_2021,guo2021modulus,guedj2025-diameter,liu2024relative}.
For further developments on this topic, we refer interested readers to~\cite{Fu-Guo-Song2020-geometric,guo2022-local,guo2022-diameter,Guo2024-diameter2,vu2024continuity,do2023log,Guedj-To-25-green,vu24-diameter,nguyen-vu24-diameter,do-nguyen-vu25-volume,zhang2026complex} and references therein.

\medskip  Our second result of this paper is to extend the one obtained by Guo, Phong, Tong, and Wang~\cite{guo2021modulus} to the case of arbitrary {big} cohomology classes, without the nef assumption. 
Precisely, we consider the complex Monge-Amp\`ere equation
\begin{equation}\label{eq: dcmae}
		(\theta+\ddc \varphi)^n=\mu,\quad\sup_X \varphi=0, 
	\end{equation}
	where $\mu=f\omega_X^n$ with $0\leq f\in L^1(\log L)^p(X{,\omega_X^n})$ for $p>n$, satisfies the compatibility condition $\mu(X)=\Vol(\theta)$, $\varphi$ is the unknown $\theta$-psh function. The $p$-Nash entropy of $f$ is defined as
    \[\|f\|_{L^1(\log)^p}:=\int_X\log(1+ f)^pf\omega_X^n. \]

Let $\textrm{Amp}(\theta)$ denote the {\em ample locus} of $\theta$ which is roughly speaking the largest Zariski open subset where $\{\theta\}$ locally behaves like
	a K\"ahler class.
\begin{theorem}\label{thm: mainthm}
    Let $(X,\omega_X)$ be an $n$-dimensional compact K\"ahler manifold and $\theta$ a closed smooth (1,1) form representing a big cohomology class. Let $\mu=f\omega_X^n$ be a positive measure with $\mu(X)=\Vol(\theta)$, where $0\leq f\in L^1(\log L)^p(X,\omega_X^n)$ for $p>n$.  
    Let $\varphi\in\mathcal{E}(X,\theta)$ be a solution to~\eqref{eq: dcmae}. 
    Then for any compact subset $U\Subset \textrm{Amp}(\theta)$,
    there exists a constant $C$ depending on $\omega_X$, $\theta$, $U$, $n$, $p$, and $\|f\|_{L^1(\log)^p}$ such that
    \begin{equation}
        |\varphi(x)-\varphi(y)|\leq \frac{C}{|\log d(x,y)|^\alpha},
    \end{equation} for all $x,y\in U$, where $d(x,y)$ denotes the geodesic distance of the two points $x,y$ with respect to $\omega_X$ and $\alpha\in(0,\frac{p-n}{n})$.
\end{theorem}

The proof of this theorem relies on the $L^\infty$-$L^1$ stability, which goes back to Guedj and Zeriahi's result~\cite{guedj2012stability}. Our approach consists in showing that the volume of the sublevel sets $(\varphi<\psi-t)$ decreases to zero in finite time by directly measuring their $\mu$-size, which is not based on Monge-Amp\`ere capacities. It relies on weak compactness
of normalized quasi-plurisubharmonic functions, basic properties of quasi-psh envelopes, and the use of auxiliary functions, which allow us to deal with big forms.

 \subsection*{Organization of the paper.}
In Section~\ref{sect: recap}, we recall the necessary preliminaries on pluripotential theory, quasi-psh envelopes, and Demailly's regularization theorem.  Section~\ref{sect: unique} is devoted to the uniqueness of solutions to the complex Monge-Amp\`ere mean field equation, proving Theorem~\ref{thm: unique}. The modulus of continuity of Monge-Amp\`ere potentials is studied in Section~\ref{sect: modulus}. 

 \subsection*{Acknowledgments.} We thank C. H. Lu for helpful discussions.
Q.-T. Dang is supported by  the Shuimu Scholar program of Tsinghua University. L. Zhang is supported by Postdoctoral Fellowship Program of China  GZC20240867.
B. Zhou is partially supported by  National Key R$\&$D Program of China 2023YFA009900 and NSFC  Grant 12271008.

   \subsection*{Declarations} 
   The authors used ChatGPT to improve the presentation of the manuscript and to assist in identifying inconsistencies in notation and minor errors. The authors take full responsibility for all mathematical claims and for any errors that remain.
\section{Preliminaries} \label{sect: recap}
	
	Throughout the paper, we let $X$ denote a compact K\"ahler manifold of dimension $n$, equipped with a K\"ahler form $\omega_X$. 
	We denote by \(dV:={\omega_X^n}/n! \)
	the volume form associated with $\omega_X$. For any measure $\mu$ on $X$, we write $L^1(\mu)$ for $L^1(X,d\mu)$.

	\subsection{Quasi-plurisubharmonic functions}

Recall that an upper semi-continuous function $ \varphi:X \rightarrow\mathbb{R}\cup\{-\infty\} $
	is called {\it quasi-plurisubharmonic} ({\it quasi-psh} for short) if it is locally the sum of a smooth and a plurisubharmonic (psh for short) function. 

    \begin{definition}
        We say that $\varphi$ is {\it $\theta$-plurisubharmonic}  ({\it $\theta$-psh} for short) if it is quasi-psh, and $$\theta_\varphi:=\theta+\ddc \varphi\geq 0$$ in the sense of currents, where ${\rm d}=
	\partial+\Bar{\partial}$ and ${\rm d}^c=\frac{i}{2}(\Bar{\partial}-\partial)$ so that $\ddc={i}\partial\Bar{\partial}$. 
    \end{definition}
	We let $\PSH(X,\theta)$ denote the set of all $\theta$-psh functions, which are not identically $-\infty$. This set is endowed with 
	the weak topology, which coincides  with 
	the $L^1(\omega_X^n)$-topology. By Hartogs' lemma, $\varphi\mapsto\sup_X\varphi$ is continuous in this weak topology,  
	so the set of $\varphi\in\PSH(X,\theta)$, with $\sup_X\varphi=0$ is compact. 
	We refer the reader to~\cite{demaillycomplex,guedj2017degenerate} for more details.

	The cohomology class $\{\theta\}$ is said to be {\em big} if the set $\PSH(X,\theta-\varepsilon\omega_X)$ is not empty for some $\varepsilon>0$.	
	We now assume that $\{\theta\}$ is big unless otherwise specified.
	By Demailly's regularization theorem \cite{demailly1992regularization}, we can find a closed positive $(1,1)$-current $T_0\in \{\theta\}$ such that $$T_0=\theta+\ddc\psi_0\geq \varepsilon_0\omega_X,$$ for some $\varepsilon_0>0$, where $\psi_0$ is a quasi-psh function with \emph{analytic singularities}, i.e., locally 
	$$	\psi_0=c\log\left[\sum_{j=1}^{N}|f_j|^2\right]+g,
	$$
	where $c\in\mathbb{Q}_{>0}$, the $f_j$'s are holomorphic functions, and $g$ is a bounded function. 
	This current $T_0$ is then locally bounded on an open Zariski subset $X\setminus\{\psi_0=-\infty\}$. 
	We thus define the {\em ample locus} $\textrm{Amp}(\theta)$ of $\theta$ as the largest such Zariski open subset, which exists by the Noetherian property of closed analytic subsets; cf.~\cite{boucksom2004divisorial}.

	Given $\varphi,\psi\in \PSH(X,\theta)$, we say that $\varphi$ is {\it less singular} than  $\psi$, and denote by $\psi\preceq\varphi$, if  there exists a constant $C$ such that $\psi\leq \varphi+C$ on $X$. We say that $\varphi,\psi$ have the {\em same singularity type}, and denote by $\varphi\simeq\psi$ if $\varphi\preceq\psi$ and $\psi\preceq \varphi$. 
	There is a
	natural least singular potential in $\PSH(X,\theta)$ given by
	\begin{align*}
		V_{\theta}:=\sup\{\varphi\in \PSH(X,\theta): \varphi\leq 0\}.
	\end{align*}
	A function $\varphi$ is said to have {\em minimal singularities} if it has the same singularity type as $V_\theta$.  In particular, $V_\theta=0$ if and only if $\theta$ is semipositive.
Moreover, $V_\theta$ is locally bounded on ${\rm Amp}(\theta)$.

\begin{theorem} [{\cite{skoda1972sous,tian1987kahler,guedj2017degenerate}}]\label{thm: Skoda/Tian} Assume $\theta\leq A\omega_X$ for fixed $A>0$.
    There exists $\alpha=\alpha(n,A)>0$ such that for all $\varphi\in\PSH(X,\theta)$,
    \begin{equation}\label{eq: alpha}
        \int_Xe^{-\alpha(\varphi-\sup_X\varphi)}\omega_X^n\leq C,
    \end{equation}
    where $C=C(n,A,\alpha)$ is independent of $\varphi$ and $\theta$. Moreover, for all $\varphi\in\PSH(X,\theta)$, 
    \[\int_X(\varphi-\sup_X\varphi)\omega_X^n\geq -C. \]
\end{theorem}

    \subsection{Demailly's regularization theorem}

	Let $$\exph\colon TX\to X,\quad \exph_x: T_x X\ni \zeta\mapsto X$$  be  the formal holomorphic part
	of the Taylor expansion of the exponential map of the Chern connection on $TX$ associated with the metric $\omega_X$.

    Following~\cite{demailly1994regularization}, for a quasi-psh function $u$,
	we define its {\em $\delta$-regularization} $\rho_\delta u=\Psi(u)(z,\delta)$ where
	\begin{equation}\label{phie}
		\Psi(u)(z,w)=\int_{\zeta\in T_{z}X}
		u\big(\exph_z(w\zeta)\big)\chi\big({|\zeta|^2_{\omega_X }}\big)\,dV_{\omega_X}(\zeta),\ \delta>0,
	\end{equation}
	where $\chi: \mathbb R_{+}\rightarrow\mathbb R_{+}$ is defined by
	\begin{center}
		$\chi(t)=\begin{cases}\frac {\eta}{(1-t)^2}\exp(\frac 1{t-1})&\ {\rm if}\ 0\leq t\leq 1,\\0&\
			{\rm if}\ t>1,\end{cases}$
	\end{center}
	with a suitable constant $\eta$ such that
	$\int_{\mathbb C^n}\chi(\Vert z\Vert^2)\,dV(z)=1$, where $dV$ denotes the Lebesgue measure on $\mathbb{C}^n$. The $\delta$-regularization $\rho_\delta u$ can be written by
	\[\rho_\delta u(z)=\frac{1}{\delta^{2n}}\int_{\zeta\in T_{z}X}
	u\big(\exph_z(\zeta)\big)\chi\bigg(\frac{|\zeta|^2_{\omega_X }}{\delta^2}\bigg)\,dV_{\omega_X}(\zeta).\]
	Intuitively, this corresponds to the familiar convolution with smoothing kernels. Actually, in the case of $\mathbb{C}^n$ endowed with the Euclidean metric, this is exactly the smoothing convolution; see~
	\cite[Remark 4.6]{demailly1994regularization}.

	The following lemma is a combination of \cite[Theorem 4.1]{demailly1994regularization} and \cite[Lemma 1.12]{berman2012regularity}.
	\begin{lemma}\label{lem kiselman}
		Let $u$ be a $\theta$-psh function and define the Kiselman-Legendre transform with level $c$ by
		\begin{equation}\label{kisleg}
			\Phi_{c,\delta}\coloneqq\inf _{0< t\leq \delta }\Big[\rho_tu(z)+K (t-\delta)
			-c\log\Big(\frac t{\delta}\Big)\Big].
		\end{equation}	
		Then, for  some positive constants $0<\delta_0<1$ and $K>1$  depending on the curvature tensor of $\omega_X$ on $X$,
		the function $\rho_tu(z)+Kt$ is increasing in $t\in (0, \delta_0]$ and
		one has the following estimate for the complex Hessian:
		\begin{equation}\label{hessest}
			\theta+dd^c \Phi_{c,\delta}\geq -(Ac+2K\delta)\,\omega_X,
		\end{equation}
		where $A$ is a lower bound of the negative part of the bisectional curvature of $\omega_X$.
	\end{lemma}

\begin{proof} The proof is almost identical to that of~\cite[Lemma 4.1]{kolodziej2019stability} (see also~\cite{DiNezza2024-envelope}), which is still valid without the boundedness of $u$. \end{proof}

    \begin{lemma}\label{lem: Jensen}
		Let $u\in\PSH(X,\theta)$.
		If
		$\rho_{\delta} u$ is the regularization of
		$u$ defined as in (\ref{phie}) then for $\delta$ small enough we have
		$$
		\int_{X} \frac{\rho_{\delta} u-u}{\delta ^2 }\omega_X^n \leq C\|u\|_{L^1(\omega_X^n)} ,
		$$
		where $C$ depends only on $n$, $X$, $\omega_X$.
	\end{lemma}
    \begin{proof}
        The proof follows verbatim from that of~\cite[Lemma 2.3]{demailly2014holder}. The difference is that since $u$ may be unbounded, the upper bound of the integral depends on the $L^1$-norm of $u$.
    \end{proof}

\subsection{Non-pluripolar product}
	\label{sect: nonpluripolar}
	
	Let $\theta^1,\ldots,\theta^n$ be closed smooth real (1,1) form representing big cohomology classes, and $\varphi_j\in\PSH(X,\theta^j)$. Following the construction of Bedford--Taylor~\cite{bedford1987fine}, it has been shown in~\cite{boucksom2010monge} that for each $k\in\mathbb{N}$,
	\[ \mathbf{1}_{\cap_j\{\varphi_j>V_{\theta^j}-k\}}(\theta^1+\ddc{\max(\varphi_1,V_{\theta^1}-k)})\wedge\cdots\wedge (\theta^n+\ddc{\max(\varphi_n, V_{\theta^n}-k)})\] is well-defined as a  Borel positive measure with finite total mass. The sequence of these measures is non-decreasing in $k$ and it converges weakly to the so-called {\em Monge--Amp\`ere product}, denoted by \[ (\theta^1+\ddc{\varphi_1})\wedge \cdots\wedge (\theta^n+\ddc{\varphi_n}),\] 
	which does not charge pluripolar sets by definition. In particular, $\theta^1=\cdots=\theta^n=\theta$ and $\varphi_1=\cdots=\varphi_n$ we obtain the non-pluripolar
	Monge--Amp\`ere measure of $\varphi$, denoted by $(\theta+\ddc\varphi)^n$ or simply by $\theta_\varphi^n$.
	
	The {\em volume} of a big cohomology class $\{\theta\}$ is given by the total mass of the non-pluripolar Monge--Ampère
	measure of $V_\theta$, i.e.,  $${\rm \Vol}(\theta):=\int_{X}\theta_{V_\theta}^n.$$
	We say that $\varphi\in\PSH(X,\theta)$ has {\it full Monge--Amp\`ere mass} if $\int_X\theta_{\varphi}^n=\Vol(\theta)$. We let \begin{align*}
		\mathcal{E}(X,\theta):=\left\{\varphi\in\PSH(X,\theta):\int_X\theta_{\varphi}^n=\Vol(\theta) \right\}
	\end{align*}
	denote the set of $\theta$-psh functions with full Monge--Amp\`ere mass.
	Note that $\theta$-psh functions with minimal singularities have full  Monge--Amp\`ere mass (see \cite[Theorem 1.16]{boucksom2010monge} for more details), but the converse is not true. 
    \medskip 


 We recall the following classical inequality.
\begin{lemma}\label{lem: maxprin}
	Let $\varphi,\psi\in\PSH(X,\theta)$. Then
	\[\theta_{\max(\varphi,\psi)}^n\geq \mathbf{1}_{\{\psi\leq\varphi \}}\theta_\varphi^n+\mathbf{1}_{\{\varphi<\psi\}}\theta_{\psi}^n. \]
    In particular, if $\varphi\leq\psi$ then $\mathbf{1}_{\{\varphi=\psi\}}\theta^n_\varphi\leq \mathbf{1}_{\{\varphi=\psi\}}\theta^n_\psi$.
\end{lemma}\begin{proof}
    See e.g.,~{\cite[Lemma 2.5]{darvas2025relative}}.
\end{proof}

 \subsection{Quasi-envelopes}
Given a measurable function $f:X\to\mathbb{R}$, we define the {\em $\theta$-psh envelope} of $f$ by
\begin{equation*}
P_\theta(f):=(\sup\{u\in\PSH(X,\theta): u\leq f\;\text{on}\, X \})^*,
\end{equation*} where the star means that we take the upper semi-continuous regularization. We use the convention that $\sup \varnothing=-\infty$.

We observe that $P_\theta(f)\in\PSH(X,\theta)$ if and only if there exists a function $u\in\PSH(X,\theta)$ that lies below $f$. We also note that $P_\theta(f+C)=P_\theta(f)+C$ for any constant $C$.

By~\cite[Chapter 9]{guedj2017degenerate}, if $f$ is finite on a non-pluripolar set, then $P_\theta(f)$ is a well-defined $\theta$-psh function on $X$. 

  \begin{lemma}[{\cite[Lemma 2.3]{DDP25-singularities}}]\label{lem: quasi-everywhere} Assume that $f$ is a measurable function.
Then
 \[P_\theta(f)=\sup\{u\in\PSH(X,\theta): u\leq f\;\text{ quasi-everywhere in}\, X \}. \]
 \end{lemma}
  \begin{theorem} \label{thm: envelope}
	Assume that $f$ is quasi-continuous, and $P_\theta(f)\in\PSH(X,\theta)$. Then $\theta_{P_\theta(f)}^n$ is concentrated
	on the contact set $\{P_\theta(f)=f \}$.
\end{theorem}
\begin{proof}
 The proof proceeds along the same lines as in \cite[Theorem 2.2]{darvas2025relative}, where Lemma~\ref{lem: quasi-everywhere} is used. For this reason, we omit the details.
\end{proof}

\begin{lemma}[{\cite[Lemma 5.11]{darvas2025relative}}]\label{lemma: ddl511}
     Let $u\in\mathcal{E}(X,\theta)$ be such that $u\leq \phi$. Let $\gamma:\mathbb{R}^+\rightarrow\mathbb{R}^+\cup\{+\infty\}$ be an increasing concave continuous function with $\gamma'\leq 1$. Then $v:=-\gamma(\phi-u)+\phi\in\PSH(X,\theta)$ and 
     \[ \theta_v^n\geq (\gamma'(\phi-u))^n\theta_u^n.\]
\end{lemma}

We recall for the later use the following domination principle. 
    \begin{proposition}[Domination principle] \label{prop: domination}
       Let $u,v$ be $\theta$-psh functions such that $u\in\mathcal{E}(X,\theta)$. If 
       \[\theta_u^n(\{u<v \})\leq c\theta_v^n(\{u<v\}) \] for some $c\in[0,1)$, then $u\geq v$.
   \end{proposition}
   \begin{proof}
       The proof is close to that of~\cite[Theorem 3.4]{darvas2025relative}. For precise a argument, we refer to~\cite[Proposition 2.11]{DangZhangZhou26-estimate-big}.
   \end{proof}
\begin{proposition}[{\cite[Proposition 4.24]{darvas2018monotonicity}}]\label{prop: unique}
  Let $u,v\in\mathcal{E}(X,\theta)$. If $e^{-\lambda v}\theta_v^n\geq e^{-\lambda u}\theta_u^n$ for some $\lambda>0$, then $u\geq v$.    
   \end{proposition}
   
\subsection{Orlicz space}
We recall some background about the Orlicz space; cf.~\cite{rao2002applications}. 
\begin{definition}
    A {\em Young} function is a convex increasing lower semicontinuous function $w:\mathbb{R}^+\rightarrow\mathbb{R}^+$ such that $w(0)=0$, $\displaystyle \lim_{t\to0}\frac{w(t)}{t}=0$ and $\displaystyle \lim_{t\to+\infty}\frac{w(t)}{t}=+\infty$. 
    Its {conjugate function} is the Legendre transform of $w(\cdot)$, i.e.,
    \[ w^*(t)=\sup_{s\geq 0}(st-w(s)).\]
\end{definition}

\begin{definition}
    Let $\mu$ be a positive measure on $X$ and let $w$ be a Young function. The {\em Orlicz space} $L^w(X,d\mu)$ is the set of all measurable functions $f$ on $X$ such that
    \[\int_X w(\varepsilon f)d\mu<+\infty \]
    for some $\varepsilon>0$. The {\em Luxembourg norm} of $f\in L^w(\mu)$ is 
    \[ \|f\|_{L^w(\mu)}:=\inf\left\{r>0:\int_X w(r^{-1}f)d\mu\leq 1 \right\}.\]
\end{definition}
We observe that $cw(t/c)\leq w(t)$ for $c\geq 1$, so if $\int_Xw(f)d\mu\leq c$ then $\|f\|_{L^w(\mu)}\leq c$. When the Young function $w(t)=t^p$ for $p\geq 1$ and $\mu=dV$ is a smooth volume form on $X$, we simply write $\|f\|_p$.

\begin{definition} \label{conditionK}
   \hypertarget{conditionK}{We say that a Young function $w$ satisfies {\em condition (K)} if there exists an increasing function $h:(0,+\infty)\to(1,+\infty)$ such that
    \[w(t)=t h(\log(1+t)), \qquad\int_1^{+\infty}h^{-1/n}(t) dt<+\infty.\]}
\end{definition}
\begin{example} We have several examples of weights satisfying condition (K):
 \begin{itemize}
        \item $w(t)=t^p$, for some $p>1$;
        \item $w(t)=t(\log(1+t))^{p}$, for some $p>n$;
        \item $w(t)=t(\log(1+t))^n(\log(1+\log(1+t)))^p$, for some $p>n$.
    \end{itemize}
\end{example}
\begin{proposition}[H\"older-Young inequality] \label{prop: HY}
Let $(\Omega,d\mu)$ be a complete measure space with $\mu(\Omega)>0$. Then for $f\in L^w(\Omega,d\mu)$, we have
    \[\int_{\Omega}fd\mu\leq \frac{2\|f\|_{L^w(\Omega,d\mu)}}{(w^*)^{-1}(1/\mu(\Omega))}, \]
    where $(w^*)^{-1}$ is the inverse function of $w^*$.
\end{proposition}



    \section{Stability and Uniqueness of twisted K\"ahler-Einstein currents}\label{sect: unique}
We extend the stability result in~\cite{guedj2018stability,lu2021stability,Lu-Phung25-uniqueness} to the case of big cohomology classes. Let $dV$ denote the smooth volume form on $X$. Let $\theta$ be a real smooth closed (1,1) form representing a big cohomology class.
For simplicity, we normalize $\Vol(\theta)=1$.




\begin{theorem} \label{thm: stability1}
     Let $f,g\in L^{w}(X,dV)$ be nonnegative functions, where $w$ satisfies \hyperlink{conditionK}{condition (K)}.  
     Assume that there is $B>0$ such that
     \[B^{-1}\leq \|f\|_1,\|g\|_1\leq B,\quad B^{-1}\leq \|f\|_w,\|g\|_w\leq B. \]
     Let $\varphi,\psi\in\mathcal{E}(X,\theta)$ be functions with minimal singularities that solve
    \[(\theta+\ddc\varphi)^n=e^{\varphi}fdV,\quad (\theta+\ddc\psi)^n=e^{\psi}gdV. \]
    Then there is $C>0$ depending on $X$, $dV$, $n$, $w$ and $B$ such that
    \[ \sup_{X}|\psi-\varphi|\leq C\||f^{1/n}-g^{1/n}|^n\|_{w}^{1/n}.\]
\end{theorem}
\begin{proof}The proof is almost identical to that of~\cite[Theorem 3.1]{Lu-Phung25-uniqueness}, which goes back to~\cite{kolodziej1996sufficient,guedj2018stability,lu2021stability}. 

Let $h_1\in\PSH(X,\theta)$, $h_2\in\PSH(X,A\omega_X)$ be such that $\sup_X h_1=\sup_X h_2=0$ and
\[(\theta+\ddc h_1)^n=bfdV,\quad (A\omega_X+\ddc h_2)^n=cfdV. \]
Since $B^{-1}\leq \|f\|_1\leq B$, it follows that $b,c$ are uniformly bounded. By Ko{\l}odziej's $L^\infty$ estimate (\cite{kolodziej1998complex,boucksom2010monge,DangZhangZhou26-estimate-big}), there exists $C_1>0$ depending on $X$, $dV$, $n$, $w$ and $B$
such that $h_1\geq V_\theta-C_1$ and $h_2\geq -C_1$. Since 
\[e^{-h_1-\log b}(\theta+\ddc h_1)^n\geq e^{-\varphi}(\theta+\ddc\varphi)^n \]
and $\theta\leq A\omega_X$ so \[e^{-h_2-\log c-C_1}(A\omega_X+\ddc h_2)^n\leq e^{-\varphi}(A\omega_X+\ddc\varphi)^n, \]
the domination principle implies $V_\theta-C_1+\log b\leq \varphi\leq C_1+\log c$. Thus there is a uniform constant $C_0$ such that $V_\theta-C_0\leq \varphi\leq C_0$. Since $\varphi-C_0\leq 0$ and is $\theta$-psh, so $\varphi-C_0\leq V_\theta$
Therefore, $|\varphi-V_\theta|\leq C_0$. By symmetry, we have $|\psi-V_\theta|\leq C_0$.  

    We assume
    \[\varepsilon:=e^{C_0/n}\||f^{1/n}-g^{1/n}|^n\|_w^{1/n}\in(0,1/2). \]
    It follows from~\cite{boucksom2010monge,DangZhangZhou26-estimate-big} that there exists a unique  $v\in\PSH(X,\theta)$ with minimal singularities solving
    \begin{equation*}
        (\theta+\ddc v)^n=\left(\frac{e^{\psi}|f^{1/n}-g^{1/n}|}{e^{C_0}\||f^{1/n}-g^{1/n}|^n\|_w^{1/n}} +a\right)dV=hdV,\quad\sup_X v=0,
    \end{equation*} where $a\geq 0$ is a normalizing constant. By the triangle inequality for the $L^w$-norm, we have $\|h\|_w\leq C$, hence $v$ has minimal singularities, i.e., $v\geq V_\theta-C_0$ by Kolodziej's $L^\infty$ estimate again. We set $$u:=(1-\varepsilon)\psi+\varepsilon v-(C_0+2n)\varepsilon.$$ By the mixed Monge–Ampère inequalities, see~\cite{dinew2009inequality}, we have
    \begin{align*}
        (\theta+\ddc u)^n&=\sum_{k=0}^n\binom{n}{k}(1-\varepsilon)^k\varepsilon^{n-k}\theta_{\psi}^k\wedge\theta_v^{n-k}\\
        &\geq \sum_{k=0}^n\binom{n}{k}(1-\varepsilon)^ke^{\psi k/n}g^{k/n} e^{\psi (n-k)/n}|f^{1/n}-g^{1/n}|^{n-k} dV\\
        &=e^\psi((1-\varepsilon)g^{1/n}+|f^{1/n}-g^{1/n}|)^ndV\\
        &\geq e^{\psi+n\log(1-\varepsilon)}fdV\geq e^ufdV
    \end{align*} noting that $\log(1-\varepsilon)\geq -2\varepsilon$. By the domination principle (Proposition~\ref{prop: unique}), we have $u\leq \varphi$. This implies that 
    \begin{align*}
        \psi \leq \varphi +(2C_0+2n)\varepsilon.
    \end{align*} By symmetry, we get the desired inequality.
\end{proof}
\begin{theorem} \label{thm: stability0}
    Let $f,g\in L^{w}(X,dV)$ be probability densities, where $w$ satisfies \hyperlink{conditionK}{condition (K)}.  
    Assume that there is $B>0$ such that \(\|f\|_w,\|g\|_w\leq B\).
    Let $\varphi,\psi\in\mathcal{E}(X,\theta)$ be functions with minimal singularities such that $\sup_X\varphi=\sup_X\psi=0$ and
    \[(\theta+\ddc\varphi)^n=fdV,\quad (\theta+\ddc\psi)^n=gdV. \]
    Then there is $C>0$ depending on $X$, $dV$, $n$, $w$ and $B$ such that
    \[ \sup_{X}|\psi-\varphi|\leq C\||f^{1/n}-g^{1/n}|^n\|_{w}^{1/n}.\]
\end{theorem}
This theorem is a slight generalization of~\cite[Theorem 3.2]{Lu-Phung25-uniqueness} to the case of big cohomology classes in which its proof can be adapted. 
The proof presented below was inspired by a talk of C. H. Lu at VIASM, we thank him for helpful discussions.
\begin{proof}

Exchanging \(\varphi\) and \(\psi\) if necessary, we may assume that $\sup_X(\varphi-\psi)\geq \sup_X(\psi-\varphi)$. Set
\[2\alpha:=\sup_X(\varphi-\psi)-\sup_X(\psi-\varphi)\geq 0. \]
Set $b=(C_1\||f^{1/n}-g^{1/n}|^n\|_{w})^{-1/n}$ for $C_1>0$.
We choose $C_1>0$ so large that both \[\int_X (2b|f^{1/n}-g^{1/n}|+g^{1/n})^ndV<2\;\text{and}\; \int_X (2b|f^{1/n}-g^{1/n}|+f^{1/n})^ndV<2.  \] 
We have either
\[ \int_{\{\varphi\leq\psi+\alpha\}} (2b|f^{1/n}-g^{1/n}|+g^{1/n})^ndV<1  \]
or
\[\int_{\{\varphi>\psi+\alpha\}} (2b|f^{1/n}-g^{1/n}|+g^{1/n})^ndV<1. \]
We assume that the first inequality holds (a similar argument applies for the second).
 If $b\leq 1$ then $C_1^{-1/n}\leq \||f^{1/n}-g^{1/n}|^n\|_{w}^{1/n}$, then by  Ko{\l}odziej's $L^\infty$ estimate (\cite{boucksom2010monge,DangZhangZhou26-estimate-big}), there exists $C_0>0$ depending on $X$, $dV$, $n$, $w$, and $B$, such that \begin{equation}\label{uni-1}
 V_\theta-C_0\leq \varphi,\psi\leq V_\theta,\end{equation} 
 we are done.
  Otherwise, we set $u_b:=P_\theta(b\varphi-(b-1)\psi)$ for $b>1$. From~\cite[Theorem 3.3, Lemma 3.3]{darvas2025relative}, we have $P_\theta(b\varphi-(b-1)V_\theta)\in\mathcal{E}(X,\theta)$.
  Since $$P_\theta(b\varphi-(b-1)\psi)\geq P_\theta(b\varphi-(b-1)V_\theta),$$ we have $u_b\in\mathcal{E}(X,\theta)$.
Denote the contact set by $\mathcal{C}:=\{u_b=b\varphi-(b-1)\psi\}$. We observe that 
\[\varphi_b:=b^{-1}u_b+(1-b^{-1})\psi\leq \varphi \] with equality on $\mathcal{C}$. The maximum principle (Lemma~\ref{lem: maxprin}) yields $\mathbf{1}_\mathcal{C}\theta^n_{\varphi_b}\leq \mathbf{1}_\mathcal{C}\theta^n_\varphi$. We set $\theta_{u_b}^n=hdV$. By the mixed Monge--Amp\`ere inequality (see~\cite{dinew2009inequality}), we have
\begin{equation*}
    \begin{split}
        \theta_{\varphi_b}^n
        &=\sum_{k=0}^n\binom{n}{k}b^{-k}(1-b^{-1})^{n-k}\theta_{u_b}^k\wedge\theta_\psi^{n-k}\\
        &\geq \sum_{k=0}^n\binom{n}{k}b^{-k}(1-b^{-1})^{n-k}h^{k/n}g^{1-k/n}\\
        &=(b^{-1}h^{1/n}+(1-b^{-1})g^{1/n})^n.
    \end{split}
\end{equation*}
Hence, on $\mathcal{C}$, we have \[h\leq (b|f^{1/n}-g^{1/n}| +g^{1/n})^n.\] Thus, we have both 
\[h\leq (2b|f^{1/n}-g^{1/n}| +g^{1/n})^n\;\text{and}\; h\leq (2b|f^{1/n}-g^{1/n}| +f^{1/n})^n.\]
It follows that
\[1=\int_X\theta^n_{u_b}=\int_{\mathcal{C}}\theta_{u_b}^n\leq \int_\mathcal{C}(b|f^{1/n}-g^{1/n}| +g^{1/n})^ndV<2. \] Since \[\int_{\mathcal{C}\cap\{\varphi\leq\psi+\alpha\}}(b|f^{1/n}-g^{1/n}| +g^{1/n})^ndV<1 \]
this yields that $\mathcal{C}\cap\{\varphi>\psi+\alpha\}$ is not empty. Let $x_b\in \mathcal{C}\cap\{\varphi>\psi+\alpha\}$. We see that
\[u_b(x_b)-V_\theta(x_b)=b(\varphi(x_b)-\psi(x_b))+\psi(x_b)-V_\theta(x_b)\geq b\alpha-C_0, \]
 hence $\sup_X u_b=\sup_X(u_b-V_\theta)\geq b\alpha -C_0$ independent of $b$. 
 Note that since $u_b-\sup_X u_b\leq V_\theta$, we have $u_b-V_\theta\leq\sup_X u_b$, hence $\sup_X(u_b-V_\theta)\leq \sup_X u_b$. Trivially, $\sup_Xu_b\leq \sup_X(u_b-V_\theta)$, therefore, $\sup_X u_b=\sup_X(u_b-V_\theta)$.
 
 {
 It follows from the convexity of the Young function $w$ and \hyperlink{conditionK}{condition (K)}
that there exists a constant $C=C(w)$ such that
\[
\|h\|_{w}
\leq
C\Bigl(
b^n\bigl\||f^{1/n}-g^{1/n}|^n\bigr\|_{w}
+\|g\|_{w}
\Bigr).
\]
By the choice of $b$, since $\|g\|_{w}\leq B$, hence
\(
\|h\|_{w}\leq C(B,C_1)\),
independently of $b$. Therefore, Ko{\l}odziej's $L^\infty$ estimate (\cite{boucksom2010monge,DangZhangZhou26-estimate-big}) yields
\[
u_b-\sup_Xu_b\geq V_\theta-C_0,
\]
where $C_0$ depends only on $X$, $dV$, $n$, $w$, and $B$.} Note that  $C_0$ in \eqref{uni-1} can be chosen sufficiently large such that it is the same here.
 Since $u_b\leq b\varphi-(b-1)\psi$, it follows that 
 \[ V_\theta-2C_0+b\alpha\leq b(\varphi-\psi)\Rightarrow \psi-\varphi+\alpha\leq 2C_0b^{-1}.\]
This gives
\[\sup_X(\psi-\varphi)\leq \frac{\sup_X(\psi-\varphi)-\sup_X(\varphi-\psi)}{2} +2C_0b^{-1},\] hence
\[\sup_X(\psi-\varphi)+\sup_X(\varphi-\psi)\leq 4C_0b^{-1}. \] Since $\sup_X\varphi=\sup_X\psi=0$ we have $\sup_X(\varphi-\psi)\geq 0$ and $\sup_X(\psi-\varphi)\geq 0$. Therefore, the above inequality yields
\[\sup_X(\varphi-\psi)\leq 4C_0b^{-1}\; \text{and}\; \sup_X(\psi-\varphi)\leq 4C_0b^{-1}. \] Hence, $\sup_X|\varphi-\psi|\leq 4C_0b^{-1}$.
 The choice of $b$ thus finishes the proof.

 In particular, if $w(t)=t^p$ for some $p>1$, then we have
 \[ \sup_X|\psi-\varphi|\leq C\|f^{1/n}-g^{1/n}\|_{np}.\]\end{proof}

We are now able to prove the uniqueness of the
solution to mean field equations with a small temperature parameter.
\begin{theorem} \label{thm: uni} 
Assume $\mu=fdV$ is a probability density on $X$, where $f\in L^p(dV)$ for some $p>1$.
    Then there exists $\lambda_0>0$ such that for all $\lambda\in(0,\lambda_0)$, the equation
    \begin{equation}\label{eq: tKE-cmae}
        (\theta+\ddc\varphi)^n=e^{-\lambda\varphi}fdV 
    \end{equation} has a unique solution $\varphi\in\PSH(X,\theta)$ with minimal singularities. Moreover, $\varphi$ is H\"older continuous on $\textrm{Amp}(\theta)$.
\end{theorem} We emphasize that the uniqueness of solutions fails for $\lambda>0$ large enough, as illustrated in an example of Fubini–Study metrics~\cite[Example 3.5]{Lu-Phung25-uniqueness}.
\begin{proof} By~\cite[Theorem B]{boucksom2010monge}, any solution of~\eqref{eq: tKE-cmae} has minimal singularities since the
right-hand side has an $L^q$ density for some $q>1$ by Demailly's openness theorem~\cite{berndtsson2015openness,guan2015proof}. The H\"older continuity follows from~\cite[Theorem D]{demailly2014holder}; see also Section~\ref{sect: modulus}.
The existence of solutions to~\eqref{eq: tKE-cmae} for sufficiently small $\lambda>0$ is proved in~\cite{Darvas-Zhang24-twisted}, using the properness of Ding twisted functional. For completeness, we give a direct proof in this regime.

Let $\PSH_0(X,\theta)\subset L^1(dV)$ be the set of $\theta$-psh functions $u$ normalized by $\sup_X u=0$. We see that $\PSH_0(X,\theta)$ is convex and compact with respect to the $L^1$-topology. Fixing $1<q<p$, it follows from the  H\"older inequality and Skoda-Tian integrability theorem that there is $\lambda>0$ such that $e^{-\lambda u}f\in L^q(dV)$. Let $\lambda_0>0$ be the supremum of all such $\lambda$. Let $\gamma>0$ such that $\gamma+\lambda<\lambda_0$. For each $u\in \PSH_0(X,\theta)$, by~\cite{boucksom2010monge,berman2013variational}, there exists a unique $v$ with minimal singularities such that 
\begin{equation}\label{eq: cmae-gamma}
    (\theta+\ddc v)^n=e^{\gamma v-(\gamma+\lambda)u}fdV. 
\end{equation}
 We show that $\Psi:\PSH_0(X,\theta)\to\PSH_0(X,\theta)$, which maps $u$ to $v-\sup_X v$, is continuous with respect to the $L^1$-topology. In fact, let $u_j$ be a sequence in $\PSH_0(X,\theta)$ converging in $L^1$ to $u$. By H\"older’s inequality and the Skoda-Tian integrability theorem, extracting a subsequence, $e^{-(\gamma+\lambda)u_j}f$ converges in $L^r$ for some $r>1$, to $e^{-(\gamma+\lambda)u}f$. By the stability property~\cite[Theorem 4.2]{guedj2018stability} (see also Theorem~\ref{thm: stability1}), $v_j$ converges uniformly to $v$ in $\PSH_0(X,\theta)$. Since $v_j$ and $v$ have minimal singularities, $$(\theta+\ddc v_j)^n\longrightarrow(\theta+\ddc v)^n$$ as $j\to+\infty$ as shown in~\cite[Theorem 2.17]{boucksom2010monge}. It follows that $v$ solves~\eqref{eq: cmae-gamma}, thus $\Psi(u_j)\to\Psi(u)$ in $\PSH_0(X,\theta)$, the continuity of $\Psi$ follows. By Schauder’s fixed point theorem, see e.g., \cite[Theorem B.2, page 302]{Taylor-bookIII}, there exists a fixed point of $\Psi$, say
$\varphi$. Thus, $\varphi+C$ solves~\eqref{eq: tKE-cmae} for some constant $C$.


It remains to prove the uniqueness.    Suppose that $\varphi$ and $\psi$ are two solutions with minimal singularities. We set $\varphi_0=\varphi-\sup_X\varphi$ and $\psi_0=\psi-\sup_X\psi$. 
     We rewrite the Monge–Ampère equations of $\varphi_0$ and $\psi_0$ as
     \[(\theta+\ddc\varphi_0)^n=e^{-\lambda\varphi_0-b}fdV,\quad(\theta+\ddc\psi_0)^n=e^{-\lambda\psi_0-c}fdV, \] where $b=\lambda\sup_X\varphi$ and $c=\lambda\sup_X\psi$. The constants $b$ and $c$ can be computed as
    \[e^{b}=\int_Xe^{-\lambda \varphi_0}fdV,\quad e^{c}=\int_Xe^{-\lambda \psi_0}fdV. \] 
By H\"older’s inequality and the Skoda–Tian estimate we have $|b|\leq C_1$ and $|c|\leq C_1$. Hence,
\[|b-c|\leq e^{C_1}|e^b-e^c|\leq \lambda C_3\sup_X|\varphi_0-\psi_0|,\] where in the last inequality we have used $|e^x-e^y|\leq |x-y|(e^x+e^y)$ for $x,y\in\mathbb{R}$.
By the stability result (Theorem ~\ref{thm: stability0}), we obtain \begin{align*}
    \sup_X|\varphi_0-\psi_0|&\leq C\sup_X|e^{-(\lambda\varphi_0+b)/n}-e^{-(\lambda\psi_0+c)/n} | \cdot\|f\|_p^{1/n}\\
    &\leq C_4\sup_X|\lambda(\varphi_0-\psi_0)+b-c|\\
    &\leq C_5\lambda\sup|\varphi_0-\psi_0|.
\end{align*}
If we choose $\lambda>0$ so small that $\lambda C_5<1$, then $\varphi_0=\psi_0$, hence $\varphi=\psi$. 
\end{proof} 
\section{Modulus of continuity of Monge-Amp\`ere Potentials}
\label{sect: modulus} In this section, we prove our main Theorem~\ref{thm: mainthm}. We establish the stability result, which is analogous to~\cite[Proposition 2.6]{eyssidieux2009singular}. 

\begin{theorem}\label{thm: stability2}
    Assume that $\mu$ is a probability measure absolutely continuous with respect to a smooth volume form $dV$, $d\mu=fdV$, where $f\in L^1(\log L)^p(X,dV)$ with $p>n$. Let $\varphi$, $\psi$ be $\theta$-psh functions such that \[\frac{1}{\Vol(\theta)}(\theta+\ddc\varphi)^n=\mu,\quad V_\theta-M_0\leq \varphi \leq V_\theta, \]
    for some positive constant $M_0>0$, and $\psi\leq 0$. Fix $\varepsilon>0$. Then there exists $C=C(\omega_X, n, p, M_0, \|f\|_{L^1(\log L)^p})>0$ such that 
    \begin{equation}\label{eq: satbility2.1}
        \sup_X(\psi-\varphi)\leq \varepsilon+ C[\mu(\{\varphi-\psi<-\varepsilon\})]^{\alpha_0},
    \end{equation} where $\alpha_0:=\frac{p-n}{pn}$. In particular, for any $0<\gamma<\frac{p}{n}-1$, we have
\begin{equation}\label{eq: satbility2.2}
    \sup_X(\psi-\varphi)\leq \frac{C}{|\log\|(\psi-\varphi)_+\|_{1}|^\gamma}.
\end{equation}
\end{theorem}

By replacing $\psi$ with $\max(\psi,\varphi)$, we may assume that $\psi\geq \varphi$. In particular, $V_\theta\geq \psi\geq V_\theta- M_0$.
We set the sub-level set $\Omega_s:
=\{\varphi<\psi-s\}$ for $s>0$. 
Fix $r>0$. 
We solve an auxiliary complex Monge-Amp\`ere equation 
\begin{equation}
    \frac{1}{\Vol(\theta)}(\theta+\ddc v_s)^n=\frac{\mathbf{1}_{\Omega_s}(\psi-\varphi-s)^r }{A_s}fdV,\; v_s\in\mathcal{E}(X,\theta),\quad\sup_X v_s=-M_0,
\end{equation} where $A_s=\int_{\Omega_s}(\psi-\varphi-s)^rfdV$. The existence of solution $v_s$ is proved in~\cite{boucksom2010monge,berman2013variational,darvas2021log}.

\begin{lemma}\label{lem: key}
    We have
    \begin{equation}\label{eq: key-equation}
       \frac{n}{n+r} (2A_s)^{-\frac{1}{n}}(\psi-\varphi-s)^{\frac{n+r}{n}}\leq \psi-v_s+\frac{n}{n+r}(2A_s)^{\frac{1}{r}}.
    \end{equation}
\end{lemma}
\begin{proof}
     Set $$\varepsilon=\left(\frac{n+r}{n}\right)^{\frac{n}{n+r}}(2A_s)^{\frac{1}{n+r}},\quad\Lambda=\frac{n}{n+r}(2A_s)^{\frac{1}{r}},\quad C_\Lambda=\varepsilon\Lambda^{\frac{n}{n+r}}$$ so that $\frac{\varepsilon n}{n+r}\Lambda^{-\frac{r}{n+r}}=1$. 
   We consider 
\[ \chi(t) := 
\begin{cases} 0, & t\leq C_\Lambda,\\[1mm]\varepsilon^{-\frac{n+r}{n}}t^{\frac{n+r}{n}}-\Lambda , & t>C_\Lambda. 
\end{cases} \] 
The function $\chi$ is continuous, convex, and increasing. 
Let $\gamma:[0,+\infty)\to[C_\Lambda,+\infty)$ be the inverse of $\chi|_{[C_\Lambda,+\infty)}$, which is concave and increasing. A direct computation gives
    \[ \gamma(t)=\varepsilon(t+\Lambda)^{\frac{n}{n+r}}\] $\gamma(0)=C_\Lambda$, $\gamma'(0)=1$ and $0<\gamma'\leq 1$.

Set $\widetilde\Omega_s:=\{C_\Lambda<\psi-\varphi-s\}$.
We see that $\widetilde\Omega_s\subset\Omega_s$. 
We set $$\phi:= \psi-\chi(\psi-\varphi-s) \; \text{and}\; u:=P_\theta(\phi).$$ We observe that $\psi$ and $\varphi$ are quasi-continuous, and $\chi$ is continuous. Hence $\phi$ is quasi-continuous.
In addition, $u$ has minimal singularities, so $u\in\mathcal{E}(X,\theta)$. We denote by $\mathcal{C}=\{u=\phi\}$ the contact set. By Theorem~\ref{thm: envelope}, we have that $\theta_u^n$ is concentrated on the contact set $\mathcal{C}$. 

Set $v:=-\gamma(\psi-u)+\psi$.
Since $u\leq \psi$ and $\gamma$ is increasing, we have
\[v\leq \psi-\gamma(\chi(\psi-\varphi-s)). \]
If $\psi-\varphi-s>C_\Lambda$, then 
\[ \gamma(\chi(\psi-\varphi-s))=\psi-\varphi-s. \] 
Otherwise, $\psi-\varphi-s\leq C_\Lambda$, then 
\[  \gamma(\chi(\psi-\varphi-s)) = \gamma(0) = C_\Lambda \geq \psi-\varphi-s. \] 
Thus, in both cases $v\leq \varphi+s$ on $X$.
Set  $\widetilde{\mathcal{C}} := \mathcal C\cap\widetilde\Omega_s$. 
On $\widetilde{\mathcal{C}}$ we have 
 $u= \psi-\chi(\psi-\varphi-s)$, so
$v=\varphi+s$ there. 
Next, we show that  \begin{equation}\label{eq: active-contact} 
\{u<v_s\}\cap\mathcal C \subset\widetilde{\mathcal{C}}. 
\end{equation} 
Indeed, let $x\in\mathcal C\setminus\widetilde{\mathcal{C}}$. 
Then $\psi-\varphi-s\leq C_\Lambda$, so $\chi(\psi-\varphi-s)=0$ and hence $u(x)=\phi(x)=\psi(x)$.  
Since $v_s\in\PSH(X,\theta)$ and $\sup_Xv_s=-M_0$, it follows that $v_s\leq V_\theta-M_0$. 
Thus $u(x)=\psi(x)\geq v_s(x)$, proving \eqref{eq: active-contact}. 
On $\widetilde{\mathcal{C}}$, we have
\[v=-(\psi-\varphi-s)+\psi=\varphi+s. \]
By Lemma~\ref{lemma: ddl511}, we have
\[ \theta_v^n \geq \bigl(\gamma'(\psi-u)\bigr)^n\theta_u^n. \] 
    It follows from Lemma~\ref{lem: maxprin} that, on the set $\{u<v_s\}$,
    \begin{equation}
        \begin{split}
        (\chi'((\psi-\varphi-s)))^{-n}\theta_u^n&=    \mathbf{1}_{\widetilde{\mathcal{C}}}(\chi'(\psi-\varphi-s))^{-n}\theta_u^n\\
        &= \mathbf{1}_{\widetilde{\mathcal{C}}}(\gamma'(\psi-u))^n\theta_u^n\\
        &\leq  \mathbf{1}_{\widetilde{\mathcal{C}}}\theta_v^n\leq \mathbf{1}_{\widetilde{\mathcal{C}}}\theta_\varphi^n,
        \end{split}
    \end{equation} since we also have $\widetilde{\mathcal{C}}\subset \{v=\varphi+s\}$. Thus, on $\{u<v_s\}$,
    \begin{align*}
       \theta_u^n&\leq (2A_s)^{-1}(\psi-\varphi-s)^{r}\theta_\varphi^n\\
        &\leq (2A_s)^{-1}(\psi-\varphi-s)^{r}\Vol(\theta)fdV \\
        &= \frac{1}{2}\theta_ {v_s}^n.
    \end{align*}
 We apply the domination principle (Proposition~\ref{prop: domination}) to obtain
    $v_s\leq u $, so $v_s\leq \psi$. We infer that on $\widetilde\Omega_s$,
    \[v_s\leq \psi-  \frac{n}{n+r} (2A_s)^{-\frac{1}{n}}(\psi-\varphi-s)^{\frac{n+r}{n}}+\left(\frac{n}{n+r}\right)(2A_s)^{\frac{1}{r}}.\]
On the other hand, on $\Omega_s\setminus\widetilde\Omega_s$, we have $0<\psi-\varphi-s\leq C_\Lambda=(2A_s)^{\frac{1}{r}}$, hence
\[\frac{n}{n+r} (2A_s)^{-\frac{1}{n}}(\psi-\varphi-s)^{\frac{n+r}{n}}\leq \frac{n}{n+r}(2A_s)^{\frac{1}{r}}. \]
Since  $-\psi\leq -V_\theta+M_0\leq -v_s$, we obtain the inequality~\eqref{eq: key-equation}. 
Since on $X\setminus\Omega_s$, the inequality is trivial and the proof is complete.
\end{proof}

\begin{proof}[Proof of Theorem~\ref{thm: stability2}]
    We follow the same arguments as in ~\cite[pages 401-403]{GuoPhongTong23-estimates}. Lemma \ref{lem: key} (with $r=1$) yields
   \begin{equation}\label{eq: MTOs}
       \int_{\Omega_s}\exp\left[ 2\beta_0\left(\frac{\psi-\varphi-s}{A_s^{{1}/{(n+1)}}}\right)^{{(n+1)}/{n}}\right]dV\leq e^{C_n\beta_0A_s}\int_{\Omega_s}e^{-C_n\beta_0v_s}dV\leq C, 
   \end{equation}
where we have chosen $\beta_0>0$ such that $2\beta_0C_n=\alpha(n,A)$, defined in Theorem~\ref{thm: Skoda/Tian} and  
\[A_s\leq \int_XM_0fdV=C_0. \] If we set
\[ v:=\beta_0\left(\frac{\psi-\varphi-s}{A_s^{{1}/{(n+1)}}}\right)^{{(n+1)}/{n}},\]
then by the H\"older-Young inequality with respect to $\eta(x)=(\log(1+x))^p$ we have 
\begin{equation*}
    \begin{split}
        v^pf&\leq \int_0^{f}\eta(x)dx+\int_0^{v^p} \eta^{-1}(y)dy\\
        &\leq f (\log(1+f))^p+v^pe^v\\
        &\leq f (\log(1+f))^p+C_pe^{2v}.
    \end{split}
\end{equation*} By integrating both sides over $\Omega_s$ and by~\eqref{eq: MTOs}, we obtain
\begin{equation*}
    \int_{\Omega_s}v^pfdV\leq \|f\|_{L^1(\log L)^p}+C.
\end{equation*} This implies
\[\int_{\Omega_s}(\psi-\varphi-s)^{\frac{p(n+1)}{n}}fdV\leq (2\beta_0)^{-p}A_s^{p/n}(\|f\|_{L^1(\log)^p}+C). \] On the other hand, H\"older inequality yields
\begin{equation}\label{eq: As}
    \begin{split}
        A_s
        &\leq \left( \int_{\Omega_s}(\psi-\varphi-s)^{\frac{p(n+1)}{n}}fdV \right)^{\frac{n}{p(n+1)}}\cdot \left( \int_{\Omega_s}fdV\right)^{1/q}\\
        &\leq A_s^{\frac{1}{n+1}} C^{\frac{n}{p(n+1)}}\left(\int_{\Omega_s}fdV\right)^{1/q},
    \end{split}
\end{equation}
where $q>1$ such that $\frac{n}{p(n+1)}+\frac{1}{q}=1$ and $C$ depends on $\omega_X$, $dV$, $\beta_0$, $n$, $p$, and $\|f\|_{L^1(\log L)^p}$. On the other hand, for $t\in[0,1]$, 
$$A_s=\int_{\Omega_s}(\psi-\varphi-s)fdV\geq t\int_{\Omega_{s+t}}fdV.$$
Together with the inequality~\eqref{eq: As}, if we set $H(s):=\int_{\Omega_s}fdV$, then we obtain
\[ tH(t+s)\leq B_0 H(s)^{1+\alpha_0},\quad\forall\, t\in [0,1], s\geq 0,\] where $\alpha_0=\frac{n+1}{nq}-1=\frac{p-n}{pn}$, $B_0$ depends only on $\omega_X$, $\beta_0$, $n$, $p$, and $\|f\|_{L^1(\log L)^p}$. 
If $H(\varepsilon)^{\alpha_0}<\frac{1}{2B_0}$, then by De Giorgi's lemma (\cite[Remark 2.5]{eyssidieux2009singular}), we have $H(s)=0$ for $s\geq S_\infty$, where
\[S_\infty=\varepsilon+\frac{2B_0 H(\varepsilon)^{\alpha_0}}{1-2^{-\alpha_0}}. \] Hence, $\theta_\varphi^n(\{\varphi<\psi-S_\infty\})=0$, 
Proposition~\ref{prop: domination} yields
$\varphi\geq\psi-S_\infty $. 
    It follows that
    \[\sup_X(\psi-\varphi)\leq S_\infty\leq \varepsilon+ C_0 [\mu(\{\varphi<\psi-\varepsilon\})]^{\alpha_0}\]
     Otherwise, $H(\varepsilon)^{\alpha_0}\geq \frac{1}{2B_0}$, so we can choose $C>0$ so that $CH(\varepsilon)^{\alpha_0}\geq 2M_0$. This finishes the proof of the statement.

\medskip 
     For the second statement, by the H\"older-Young inequality (Proposition~\ref{prop: HY}), we have
     \begin{equation*}
         \begin{split}
             \mu(\{\varphi<\psi-\varepsilon\})&\leq \frac{C\|f\|_{L^1(\log L)^p}}{({w^*})^{-1}\left( \frac{1}{\Vol(\Omega_\varepsilon)}\right)},
         \end{split}
     \end{equation*} where $(w^*)^{-1}(t)\gtrsim(\log t)^p$. 
     Therefore, we obtain
     \begin{align*}
         \sup_X(\psi-\varphi)&\leq \varepsilon+\frac{C}{[-\log\Vol(\{\varphi<\psi-\varepsilon\})]^{p\alpha_0}}\\
         &\leq \varepsilon+ {C}\left[-\log\frac{\|(\psi-\varphi)_+\|_1}{\varepsilon}\right]^{-p\alpha_0}.
     \end{align*}Choosing $\varepsilon:=[-\log\|(\psi-\varphi)_+\|_1]^{-\gamma}$ for $\gamma\in (0,p\alpha_0)$, the conclusion follows.
\end{proof}
\begin{remark} When $f\in L^p$ for $p>1$, arguing as in Theorem~\ref{thm: stability2}, 
we obtain the following estimate:
 \begin{equation}
        \sup_X(\psi-\varphi)\leq \varepsilon+ C[\mu(\{\varphi<\psi-\varepsilon\})]^{\alpha},
    \end{equation} for any $\alpha\in(0,\frac{1}{n})$. By H\"older's inequality and by choosing $\varepsilon:=\|(\psi-\varphi)_+\|^\gamma$, we eventually have
    \[ \sup_X(\psi-\varphi)\leq C\|(\psi-\varphi)_+\|_1^\gamma\]
for any $\gamma\in(0,\frac{1}{nq+1})$ with $\frac{1}{p}+\frac{1}{q}=1$. 
This recovers the stability estimate of Guedj and Zeriahi
\cite[Proposition 5.2]{guedj2012stability}, which is a key ingredient in the proof of
\cite[Theorem D]{demailly2014holder}.
\end{remark}
\begin{remark}
    We can strengthen~\cite[Theorem 1.2]{zhang2026complex} without the nefness condition. Indeed, by using the inequality~\eqref{eq: MTOs}, we can apply Moser's iteration as in~\cite[Steps 5-8]{zhang2026complex} to obtain the stability estimate.
\end{remark}
\begin{proof}[Proof of Theorem~\ref{thm: mainthm}]
The proof follows the same arguments as in~\cite[Theorem D]{demailly2014holder}. We provide the proof here for the sake of completeness.

The solution $\varphi$
is unique up to an additive constant and has minimal singularities (see~\cite{kolodziej1998complex,boucksom2010monge,DangZhangZhou26-estimate-big}).
We can thus assume, without loss of generality, that $-M_0+V_\theta\leq \varphi\leq V_\theta$.

Fix $\gamma\in (0,\frac{p}{n}-1)$.
Set 
\[\varphi_{c,\delta}:= \frac{Ac+2K\delta}{\varepsilon_0}\psi_0+\left( 1-\frac{Ac+2K\delta}{\varepsilon_0}\right)\Phi_{c,\delta},\] where $0<\delta\leq\delta_0$, $c>0$ and $\Phi_{c,\delta}$ is the Kiselman--Legendre transform.
 In the following arguments, we choose $c>0$ so that $Ac+2K\delta=\varepsilon_0b(\delta)$, where $b(\delta)=(-\log\delta)^{-\gamma}$ and write $\varphi_\delta$ instead of $\varphi_{c,\delta}$.  
 We observe that $$\varphi_\delta\leq \Phi_{c,\delta}\leq\rho_\delta\varphi\leq 0$$ for $0<\delta\leq \delta_0$ sufficiently small. By Lemma \ref{lem kiselman}, $\theta+dd^c\Phi_{c,\delta}\geq -\varepsilon_0b(\delta)\omega_X$.
 Hence
 \[\theta+dd^c\varphi_\delta\geq \varepsilon_0b(\delta)\omega_X-(1-b(\delta))\varepsilon_0b(\delta)\omega_X=b(\delta)^2\varepsilon_0\omega_X,\]
 which implies $\varphi_\delta$ is $\theta$-psh.
 By the stability result (Theorem~\ref{thm: stability2}), we obtain
\begin{align*}
    \sup_X(\varphi_\delta-\varphi)&\leq C_0(-\log\|(\varphi_\delta-\varphi)_+\|_1)^{-\gamma},
\end{align*} where $C_0=C_0(\omega_X, n, p, M_0, \|f\|_{L^1(\log L)^p})>0$.
By Lemma \ref{lem kiselman} again, $\rho_t\varphi+Kt$ is increasing. Hence,
$\rho_\delta\varphi-\varphi\geq -K\delta$. Then by Lemma~\ref{lem: Jensen},
we have $\|(\rho_\delta\varphi-\varphi)_+\|_1\leq C\delta^2\|\varphi\|_1$.
Since $(\varphi_\delta-\varphi)_+\leq(\rho_\delta\varphi-\varphi)_+ $, we obtain
\[\sup_X(\varphi_\delta-\varphi)\leq \frac{C_1}{(-\log \delta)^{\gamma}},\]
where $C_1$ depends on $C_0$, $K$, $M_0$ and the Chern curvature of $\omega_X$. 

Fix the point $z\in \textrm{Amp}(\theta)$. 
Then the minimum in the definition of $\Phi_{c,\delta}$ is realized at $t_0=t_0(z)$. 
Therefore, the last inequality yields, at the point $z$,
\[b(\delta)(\psi_0-\varphi)+(1-b(\delta))(\rho_{t_0}\varphi+K(t_0-\delta)-\varphi-c\log(t_0/\delta))\leq C_1(-\log \delta)^{-\gamma}. \]
Since $\rho_{t_0}\varphi+Kt_0-\varphi\geq 0$ we infer that
\[ c(1-b(\delta))\log\frac{t_0}{\delta}\geq b(\delta)(\psi_0(z)-V_\theta(z)-C_2),\]
for $\delta\leq \delta_0$ sufficiently small, for $C_2=C_2(C_1,K)>0$.
Combining with $$c=\varepsilon_0 A^{-1}b(\delta)-2KA^{-1}\delta\geq\frac{\varepsilon_0}{2A} b(\delta)$$ if $\delta\leq \delta_0$ sufficiently small, then one gets that
\[ t_0\geq \delta\kappa\;\; \text{with}\; \kappa(z)=\exp \left( \frac{2A(\psi_0(z)-V_\theta(z)-C_2)}{\varepsilon_0(1-b(\delta_0))}\right).\] Since $t\mapsto\rho_t\varphi+Kt$ is increasing and $t_0(z)\geq \delta\kappa(z)$ it follows that
\begin{align*}
\rho_{\delta\kappa}\varphi(z)+K\delta\kappa\leq \rho_{t_0}\varphi(z)+Kt_0\leq\Phi_{c,\delta}(z)+K\delta= \frac{\varphi_\delta(z)-b(\delta)\psi_0}{1-b(\delta)}+K\delta.
\end{align*}
By the stability result again, we obtain
\[ \varphi_\delta-b(\delta)\psi_0\leq \varphi+b(\delta)(C_2-\psi_0),\]
using that $\varphi\leq 0$, hence
\begin{align*}
    \rho_{\delta\kappa(z)}\varphi(z)-\varphi(z)&\leq \frac{b(\delta)}{1-b(\delta)} (C_2-\psi_0(z))\\
    &\leq \frac{b(\delta)}{1-b(\delta_0)}(C_2-\psi_0(z)).
\end{align*} Replacing $\delta$ by $\delta\kappa(z)^{-1}$, one gets
\[ \rho_\delta\varphi(z)-\varphi(z)\leq (1-b(\delta_0))^{-1}(C_2-\psi_0(z))\exp \left( \frac{4A\gamma(C_2+V_\theta(z)-\psi_0(z))}{\varepsilon_0(1-b(\delta_0))}\right)b(\delta). \]
It follows that for any compact set $U\Subset \textrm{Amp}(\theta)$ there exists a constant $C$ depending on $U$ such that for any $z\in U$,
\[\rho_\delta\varphi(z)-\varphi(z)\leq C(-\log\delta)^{-\gamma}. \]
Lemma~\ref{lem:regularization-to-modulus} thus completes the proof. 
\end{proof}

\begin{lemma}\label{lem:regularization-to-modulus} 
Let $U\Subset {\rm Amp}(\theta)$ be a compact subset. Assume that there exists
a compact set $U'$ with
\(
U\Subset U'\Subset {\rm Amp}(\theta),
\)
and constants $\alpha>0$, $\delta_1>0$, and $C_U>0$ such that
\[
    \rho_\delta \varphi(x)-\varphi(x)
    \leq \frac{C_U}{|\log \delta|^\alpha},
    \qquad
    x\in U',\quad 0<\delta\leq \delta_1.
\]
Then for any $x,y\in U$, there exist constants $\delta_2>0$ and $C>0$ such that
\[
    |\varphi(x)-\varphi(y)|
    \leq
    \frac{C}
    {\left|\log d_{\omega_X}(x,y)\right|^\alpha}
\]
whenever $
    d_{\omega_X}(x,y)\leq \delta_2$.

\end{lemma}

\begin{proof} 
We follow the standard argument of \cite[Lemma 4.4]{lu2021stability}
(see also \cite{Zeriahi20-continuity}); we only sketch the main steps. 
By the local properties of Demailly's regularization and the submean
inequality, the assumption
\[
\rho_\delta\varphi-\varphi
\leq C|\log\delta|^{-\alpha}
\qquad\text{on }U'
\]
yields, for $\delta>0$ sufficiently small, an estimate for the local
oscillation
\[
m_\varphi(\delta)
:=
\sup\left\{
\varphi(y)-\varphi(x)
\,:\,
x,y\in U,\ 
d_{\omega_X}(x,y)\leq\delta
\right\}
\]
of the form
\begin{equation}\label{eq:oscillation-iteration}
m_\varphi(\delta)
\leq
\frac{3}{4}m_\varphi(\gamma\delta)
+
C|\log\delta|^{-\alpha},
\end{equation}
where $\gamma>1$ and $C>0$ are independent of $\delta$.
Iterating \eqref{eq:oscillation-iteration} $N$ times, and using the boundedness of $\varphi$ on
compact subsets of $\text{Amp}(\theta)$, we obtain
\[
m_\varphi(\delta)
\leq
C|\log\delta|^{-\alpha}.
\]
It follows that
\[
|\varphi(x)-\varphi(y)|
\leq
\frac{C_U}
{|\log d_{\omega_X}(x,y)|^\alpha},
\qquad x,y\in U,
\]
which concludes the proof.  \end{proof}

    We can adapt our argument to obtain a slight generalization of Theorem~\ref{thm: mainthm}, which also generalizes~\cite[Theorem 1.6]{guedj2025-diameter}.
    \begin{theorem}\label{mainthm}
        Let $\mu=fdV$ be a positive measure with $\mu(X)=\Vol(\theta)$, where $0\leq f\in L^1(\log L)^n(\log \log L)^p(X,dV)$ for $p>n$.  
        Let $\varphi\in\mathcal{E}(X,\theta)$ be a solution to 
        \[(\theta+\ddc\varphi)^n=\mu,\quad\sup_X\varphi=0. \]
        Then for any compact subset $K\Subset \textrm{Amp}(\theta)$,
        there exists a constant $C>0$ depending on $\omega_X$, $K$, $n$, $p$, and $\|f\|_{L^1(\log)^n(\log \log L)^p}$ such that
        \begin{equation}
            |\varphi(x)-\varphi(y)|\leq \frac{C}{|\log(-\log d(x,y))|^\alpha},
        \end{equation} for all $x,y\in K$, where $d(x,y)$ denotes the geodesic distance of the two points $x,y$ with respect to $\omega_X$ and $\alpha\in(0,\frac{p-n}{n})$.
    \end{theorem}

\begin{proof}
The proof is a slight modification of that of Theorem~\ref{thm: mainthm},
so we only indicate the points where the argument differs.
For $0\geq\psi\in\PSH(X,\theta)$, we set
\[
\Omega_s:=\{\varphi<\psi-s\},
\qquad
H(s):=\mu(\Omega_s).
\]
As in the proof of Theorem~\ref{thm: stability2}, Lemma~\ref{lem: key}
and the Skoda--Tian integrability theorem (Theorem~\ref{thm: Skoda/Tian}) yield
\[
\int_{\Omega_s}
\exp\left[
2\beta_0
\left(
\frac{\psi-\varphi-s}
{A_s^{1/(n+1)}}
\right)^{\frac{n+1}{n}}
\right]dV
\le C,
\]
where
\[
A_s:=\int_{\Omega_s}(\psi-\varphi-s)f\,dV.
\]
By the H\"older--Young inequality with the Young function
corresponding to
\[
f\in L^1(\log L)^n(\log\log L)^p,
\]
one obtains
\begin{equation}\label{eq:loglog-degiorgi}
tH(s+t)
\le
\frac{B_0 H(s)}
{\bigl[\log(1/H(s))\bigr]^{p/n}},
\qquad
0<t\le1,\quad s\ge0,
\end{equation}
where $B_0>0$ is uniformly controlled.
Set
$
g(s):=-\log H(s)$.
Then \eqref{eq:loglog-degiorgi} yields
\begin{equation}\label{eq: De Giorgi}
    g(s+t)
\ge
g(s)+\log t-\log B_0+\frac{p}{n}\log g(s).
\end{equation}
If $eB_0[g(\varepsilon)]^{-p/n}\geq 1$, we are done. Otherwise, starting from $s_0=\varepsilon$, define inductively $(s_j)_{j\in\mathbb N}$ by induction setting 
\[s_0=\varepsilon,\quad s_{j+1}=s_j+eB_0[g(s_j)]^{-p/n}. \]
Since $[g(s)]^{-p/n}$ is decreasing, we have $$0\leq t_j=s_{j+1}-s_j=eB_0[g(s_j)]^{-p/n}\leq eB_0[g(\varepsilon)]^{-p/n}\leq 1.$$ This allows to use~\eqref{eq: De Giorgi} to obtain $g(s_j)\geq j+g(\varepsilon)\geq j$. Now, we have
\[s_\infty=\lim_{j}s_j=s_0+\sum_{j}(s_{j+1}-s_j)\leq s_0+eB_0\sum_j (g(\varepsilon)+j)^{-p/n}\leq s_0+\frac{eB_0ng(\varepsilon)^{1-p/n}}{p-n}. \]
It follows that $g(s)=+\infty$ so $H(s)=0$ for $s\ge s_\infty$. Hence, $\theta_\varphi^n(\{\varphi<\psi-s_\infty\})=0$, and by Proposition~\ref{prop: domination}, $\varphi\geq \psi-s_\infty$. 
Consequently,
\[
\sup_X(\psi-\varphi)
\le
\varepsilon+
\frac{C}
{\bigl[-\log H(\varepsilon)\bigr]^{(p-n)/n}}.
\]
By the H\"older-Young inequality (Proposition~\ref{prop: HY}), we obtain
   \[\sup_X(\psi-\varphi)\leq \varepsilon+\frac{C}{[\log (\log(\varepsilon/\|(\psi-\varphi)\|_1))]^{(p-n)/n}}. \]
  Choosing $\varepsilon:=[\log (-\log(\|(\psi-\varphi)\|_1))]^{-\gamma}$ for any $\gamma\in (0,\frac{p-n}{n})$, we obtain
  \[\sup_X(\psi-\varphi)\leq \frac{B}{[\log (-\log(\|(\psi-\varphi)\|_1))]^{\gamma}}. \]

The remainder of the proof is identical to the
Demailly--Kiselman argument used in the proof of
Theorem~\ref{thm: mainthm}, with $
(-\log\delta)^{-\gamma}$ replaced by $\bigl[\log(-\log\delta)\bigr]^{-\gamma}$.
We obtain, for every $K\Subset\textrm{Amp}(\theta)$,
\[
\rho_\delta\varphi-\varphi
\le
\frac{C_K}
{\bigl[\log(-\log\delta)\bigr]^\gamma}.
\]
Applying Lemma~\ref{lem:regularization-to-modulus}, with $|\log\delta|^{-\gamma}$ replaced by $|\log(-\log\delta)|^{-\gamma}$, yields
\[
|\varphi(x)-\varphi(y)|
\le
\frac{C_K}
{\bigl|\log(-\log d_{\omega_X}(x,y))\bigr|^\gamma},
\qquad x,y\in K,
\]
which completes the proof.
\end{proof}

    \bibliographystyle{alpha}
	\bibliography{bibfile}	
\end{document}